\documentclass[12pt,a4paper]{article}

\usepackage[leqno]{amsmath}
\usepackage{amssymb,amsthm,upref,amscd}
\usepackage[T1]{fontenc}
\usepackage{times}
\usepackage{cite}
\usepackage{amsfonts}
\usepackage[colorinlistoftodos]{todonotes}
\usepackage{setspace}
\usepackage{comment}

\usepackage{color}
\usepackage{graphicx}
\usepackage{wrapfig}

\usepackage[usenames,dvipsnames]{pstricks}
\usepackage{epsfig}

\newtheorem{thm}{Theorem}[section]
\newtheorem{lem}{Lemma}[section]

\theoremstyle{definition}
\newtheorem{definition}{Definition}[section]

\newtheorem*{notation}{Notation}
\newtheorem*{ack}{Acknowledgements}
\theoremstyle{rmq}
\newtheorem{rmq}{Remark}[section]
\numberwithin{equation}{section}

\newcommand{\N}{{\mathbb N}}
\newcommand{\Z}{{\mathbb Z}}
\newcommand{\R}{{\mathbb R}}

\definecolor{blu}{rgb}{0,0,1}

\newcommand{\be}{\beta}

\newcommand{\la}{\lambda}

\newcommand{\De}{\Delta}

\newcommand{\eps}{\varepsilon}

\newcommand{\beq[1]}{\begin{equation}\label{eq:#1}}
\newcommand{\eeq}{\end{equation}}

\begin{document}
\title{Existence and orbital stability of standing waves to nonlinear Schr\"odinger system with partial confinement}

\author{{Tianxiang Gou}}

\date{}
\maketitle

\begin{abstract}
We are concerned with the existence of solutions to the following nonlinear Schr\"odinger system in $\R^3$,
\begin{equation*}
\left\{
\begin{aligned}
-\De u_1  + (x_1^2+x_2^2)u_1&= \la_1 u_1 + \mu_1 |u_1|^{p_1 -2}u_1
            + \be r_1|u_1|^{r_1-2}u_1|u_2|^{r_2}, \\
-\De u_2 + (x_1^2+x_2^2)u_2&= \la_2 u_2 + \mu_2 |u_2|^{p_2 -2}u_2
            +\be  r_2 |u_1|^{r_1}|u_2|^{r_2 -2}u_2,
\end{aligned}
\right.
\end{equation*}
under the constraint
\begin{align*}
\int_{\R^3}|u_1|^2 \, dx = a_1>0,\quad \int_{\R^3}|u_2|^2 \, dx = a_2>0,
\end{align*}
where $\mu_1, \mu_2, \beta >0, 2 <p_1, p_2 < \frac{10}{3}$, $r_1, r_2>1, r_1 + r_2 < \frac{10}{3}$. In the system, the parameters $\lambda_1, \lambda_2 \in \R$ are unknown and appear as the associated Lagrange multipliers. Our solutions are achieved as global minimizers of the underlying energy functional subject to the constraint. Our purpose is to establish the compactness of any minimizing sequence up to translations. As a by-product, we obtain the orbital stability of the set of global minimizers.
\end{abstract}

{\bf Keywords}: Nonlinear Schr\"odinger system; Standing waves; Global minimizers; Orbital stability; Rearrangement


\section{Introduction}\label{sec:intro}

We are interested in the following time-dependent nonlinear Schr\"odinger system in $\R^3$,
\begin{equation}\label{sys}
\begin{cases}
- i \partial_t \varphi_1 + (x_1^2+x_2^2)\varphi_1 = \Delta \varphi_1 + \mu_1 |\varphi_1|^{p_1-2}\varphi_1
                   +\beta r_1 |\varphi_1|^{r_1-2}\varphi_1|\varphi_2|^{r_2},\\
- i \partial_t \varphi_2 + (x_1^2+x_2^2)\varphi_2= \Delta \varphi_2 + \mu_2 |\varphi_2|^{p_2-2}\varphi_2
                   +\beta r_2 |\varphi_1|^{r_1}|\varphi_2|^{r_2-2}\varphi_2.
\end{cases}
\end{equation}
This system governs various physical phenomena, such as the Bose-Einstein condensates with multiple states, or the propagation of mutually incoherent waves packets in nonlinear optics, see for example \cite{AA, EG, Fr, M}. In this system, the functions $\varphi_1 ,\varphi_2$ are corresponding condensate amplitudes, the parameters $\mu_i$ and $\beta$ are intraspecies and interspecies scattering length, describing interaction of the same state and different states, respectively. The positive sign of $\mu_i$ (and $\beta$) represents attractive interaction, and the negative one represents repulsive interaction.

From a physical and mathematical point of view, a fundamental issue to consider \eqref{sys} consists in studying standing waves. Here by standing waves to \eqref{sys} we mean solutions of the form
$$
\varphi_1(t,x) = e^{-i\la_1 t} u_1(x), \quad \varphi_2(t,x) = e^{-i\la_2 t} u_2(x)
$$
for $\lambda_1, \lambda_2 \in \R$. This ansatz then gives rise to the following stationary nonlinear Schr\"odinger system satisfied by $u_1$ and $ u_2$,
\begin{equation} \label{system}
\left\{
\begin{aligned}
-\De u_1  + (x_1^2+x_2^2)u_1&= \la_1 u_1 + \mu_1 |u_1|^{p_1 -2}u_1
            +\be  r_1 |u_1|^{r_1-2}u_1|u_2|^{r_2}, \\
-\De u_2 + (x_1^2+x_2^2)u_2&= \la_2 u_2 + \mu_2 |u_2|^{p_2 -2}u_2
            + \be r_2 |u_1|^{r_1}|u_2|^{r_2 -2}u_2.
\end{aligned}
\right.
\end{equation}

Observe that the $L^2$-norm of the solution to the Cauchy problem of \eqref{sys} is conserved along time, namely for any $t>0$,
$$
\int_{\R^3}|\varphi_i(t, \cdot)|^{2} \, dx =\int_{\R^3}|u_i|^{2} \, dx \ \ \text{for} \ \ i=1,2.
$$
Furthermore, the $L^2$-norm of solution often denotes the number of particles of each component in Bose-Einstein condensates, or the power supply in nonlinear optics. Thus from these physical relevance, it is of great interest to seek for solutions to \eqref{system} having prescribed $L^2$-norm. More precisely, for given $a_1, a_2>0$, to find $(\lambda_1, \lambda_2) \in \R^2$ and $(u_1, u_2) \in H \times H$ satisfying \eqref{system} together with normalized condition,
\begin{align} \label{mass}
\int_{\R^3}|u_1|^2 \, dx=a_1, \quad \int_{\R^3}|u_2|^2 \, dx =a_2.
\end{align}
Here $H$ is the Sobolev space defined by
$$
H:=\{u \in H^1(\R^3): \int_{\R^3}(x_1^2 + x_2^2)|u|^2 \, dx <\infty\},
$$
equipped with the norm
$\|u\|_H^2:=\|u\|_{\dot H}^2+ \|u\|_{2}^2$, where
$$
\|u\|_{\dot H}^2:= \int_{\R^3}|\nabla u|^2 + (x_1^2 + x_2^2)|u|^2\, dx.
$$
Such a solution $(u_1, u_2) \in H \times H$ is obtained as a critical point of the energy functional $J: H \times H \to \R$ given by
\begin{align*}
J(u_1,u_2)
 &:= \frac12\int_{\R^3} |\nabla u_1|^2 + |\nabla u_2|^2+
     \left( x_1^2 + x_2^2 \right) \left(|u_1|^2 + |u_2|^2\right) \, dx\\
 &   -  \frac{\mu_1}{p_1}\int_{\R^3} |u_1|^{p_1} \, dx
     - \frac{\mu_2}{p_2}  \int_{\R^3} |u_2|^{p_2} \, dx
     - \be  \int_{\R^3}|u_1|^{r_1}|u_2|^{r_2} \, dx,
\end{align*}
on the constraint $S(a_1, a_2):=S(a_1) \times S(a_2)$, where
$$
S(a):= \{u \in H: \int_{\R^3}|u|^2 \, dx =a > 0\},
$$
and the parameters $\lambda_1, \lambda_2$ are then determined as Lagrange multipliers.

The main motivation of studying solutions to \eqref{system}-\eqref{mass} stems from \cite{Be}, where the authors concerned the existence and some quantitative properties of normalized solutions to a single equation with the partial confinement $x_1^2 + x_2^2$ in the mass supercritical regime. See also \cite{ACS, Oh1} for the associated physical interest and the study of the Schr\"odinger equation with the partial confinement. However, so far the research of normalized solutions to a system with the partial confinement is still open. Thus the current purpose of the paper is to investigate the existence of solutions to \eqref{system}-\eqref{mass} under the following assumption
\begin{enumerate}
\item [$(H_0)$] $\mu_1, \mu_2, \beta>0, 2 <p_1, p_2 < \frac{10}{3}, r_1, r_2>1, r_1 + r_2 < \frac{10}{3}.$
\end{enumerate}
Note first that under the assumption $(H_0)$, the energy functional $J$ restricted to $S(a_1, a_2)$ is bounded from below. Thus it is natural to introduce the following minimization problem
\begin{align} \label{min}
m(a_1, a_2):=\inf_{(u_1, u_2) \in S(a_1, a_2)}J(u_1, u_2).
\end{align}
Clearly, minimizers to \eqref{min} are critical points of the energy functional $J$ restricted to $S(a_1, a_2)$, then solutions to \eqref{system}-\eqref{mass}. We now state our main result.

\begin{thm} \label{existence}
Assume $(H_0)$. Then any minimizing sequence to \eqref{min} is compact in $H \times H$ up to translations. In particular, there exists a solution to \eqref{system}-\eqref{mass} as a minimizer to \eqref{min}.
\end{thm}

We point out that if one replaces the partial confinement $x_1^2+x_2^2$ in the energy functional $J$ by a trapping potential $x_1^2+x_2^2+ x_3^2$, and in this situation we denote by $\tilde{H}$ the corresponding Sobolev space induced by $x_1^2+x_2^2+ x_3^2$, namely
$$
\tilde{H}:=\{u \in H^1(\R^3): \int_{\R^3}(x_1^2 + x_2^2 + x_3^2)|u|^2 \, dx <\infty\},
$$
then the compactness of any minimizing sequence is a simple consequence of the fact that the embedding $\tilde{H} \hookrightarrow L^2(\R^3)$ is compact. However, in our case such a compact embedding is violated, thus it is more difficult to discuss the compactness of any minimizing sequence to \eqref{min}.

We now present the strategy to prove Theorem \ref{existence}. Suppose that $\{(u_1^n, u_2^n)\} \subset S(a_1, a_2)$ is an arbitrary minimizing sequence to \eqref{min}. Under the assumption $(H_0)$, one first finds that $\{(u_1^n, u_2^n)\}$ is bounded in $H \times H$. In order to see the compactness of $\{(u_1^n, u_2^n)\}$ in $H \times H$, we shall follow the Lions concentration compactness principle \cite{Li1, Li2}, thus it suffices to exclude the possibilities of vanishing and dichotomy of any minimizing sequence. Indeed, vanishing can be done by taking advantage of the classical Lions concentration compactness Lemma \cite[Lemma I.1]{Li2} and key ingredients proposed in \cite{Be, BeVi}.

At this point, it remains to rule out the possibility of dichotomy. For this, the heuristic argument is to establish the following strict subadditivity inequality
\begin{align} \label{subadditivity}
m(a_1, a_2) < m(b_1, b_2)+ m(a_1-b_1, a_2-b_2),
\end{align}
where $0\leq b_i \leq a_i$ for $ i=1, 2$, $(b_1, b_2) \neq (a_1, a_2)$ and $(b_1, b_2) \neq (0, 0)$. When dealing with minimization problems only having one $L^2$-norm constraint, in general the associated strict subadditivity inequality can be achieved with the aid of scaling techniques, see for instance \cite{BeSi, CJS, Sh1}. However, when it comes to minimization problems with multiple constraints, scaling techniques are not applicable any more, which causes more involved to establish strict subadditivity inequality under this circumstance. In this direction, for the special case that the dimension $n=1$, in \cite{Bh1, Bh, NW}, the authors established the corresponding strict subadditivity inequality with the help of \cite[Lemma 2.10]{AB}, and we should point out that the original idea comes from \cite{By}, see also \cite{NgWa1, NgWa2, Oh, Ga} for the relevant research. However, coming back to our minimization problem \eqref{min}, it seems hard to check \eqref{subadditivity} by using the ingredients mentioned before. For this reason, we make use of the coupled rearrangement arguments presented in \cite{Sh} to prevent dichotomy from happening, which is independent of proving that \eqref{subadditivity} holds true. The remarkable feature of the coupled rearrangement arguments is that it can strictly reduce the sum of gradient $L^2$-norm of two functions, for more details see \eqref{strict}.

Let us mention \cite{GoJe}, where the authors also studied the compactness of any minimizing sequence through the coupled rearrangement arguments from \cite{Ik}. We would like to highlight some differences of the strategy proposed in the present paper to discuss the compactness of any minimizing sequence with the one in \cite{GoJe}. Firstly, because of the presence of the partial confinement $x_1^2+x_2^2$, here we need to employ the coupled rearrangement of functions with respect to the variable $x_3$, see the definition \eqref{defcoupled}, which makes the proof of the compactness more delicate. Secondly, in \cite{GoJe}, where the infimum of the associated global minimization problem is strictly negative, thus the convergence of minimizing sequence in $L^2(\R^N) \times L^2(\R^N)$ benefits from its convergence in $L^p(\R^N) \times L^p(\R^N)$ for $2<p<6$. However, in our situation, it is unknown if there holds $m(a_1, a_2) <0$, which enables the verification of the convergence of any minimizing sequence in $L^2(\R^3) \times L^2(\R^3)$ more complex.

Defining
$$
G(a_1,a_2) := \{(u_1, u_2) \in S(a_1, a_2): J(u_1, u_2) = m(a_1, a_2)\},
$$
we now show the orbital stability of minimizers to \eqref{min}.

\begin{thm}\label{stable}
Assume $(H_0)$ and the local existence to the Cauchy problem of \eqref{sys} hold. Then the set $G(a_1, a_2)$ is orbitally stable. Namely for any $\eps>0$ there exists $\delta>0$
so that if $(\varphi_{1, 0}, \varphi_{2, 0}) \in H \times H$ satisfies
\begin{equation*}
\inf_{(u_1, u_2)\in G(a_1, a_2)} \|(\varphi_{1, 0}, \varphi_{2, 0}) - (u_1,u_2)\|_{H \times H}\leq \delta,
\end{equation*}
then
\begin{equation*}
\sup_{t  \in (0, T)} \inf_{(u_1, u_2)\in G(a_1, a_2)}
\|(\varphi_1(t), \varphi_2(t)) - (u_1, u_2)\|_{H \times H} \leq \eps,
\end{equation*}
where $(\varphi_1(t), \varphi_2(t))$ is the solution to the Cauchy problem of \eqref{sys}
with the initial datum $(\varphi_{1, 0}, \varphi_{2, 0})$, $T$ denotes the maximum existence time of the solution, and $\|\cdot\|_{H \times H}$ stands for the norm in the Sobolev space $H\times H$.
\end{thm}

Based upon Theorem \ref{existence}, the proof of Theorem \ref{stable} is a direct adaption of the classical arguments in \cite{CaLi},thus we shall not provide it.

This paper is organized as follows. In Section \ref{pre}, we recall some preliminary results concerning rearrangements, namely the Steiner the rearrangement and the coupled rearrangement. In Section \ref{scal}, we deal with an associated one constraint minimization problem. Finally, Section \ref{proof} is devoted to the proof of Theorem \ref{existence}.

\begin{ack}
The author wants to express his appreciation to Louis Jeanjean for useful discussion. The author is supported by the China Scholarship Council.
\end{ack}

\begin{notation}
In this paper we write $L^p(\R^3)$ the usual Lebesgue space endowed with the norm
$$
\|u\|_p^p := \int_{\R^3}|u|^p\,dx,
$$
and $H^1(\R^3)$ the usual Sobolev space endowed with the norm
$$
\|u\|^2 := \int_{\R^3}|\nabla u|^2+|u|^2 \, dx.
$$
We denote by $'\rightarrow'$ and $'\rightharpoonup'$ the strong convergence and weak convergence in corresponding space, respectively.
\end{notation}

\section{Preliminary results} \label{pre}

Firstly, notice that under the assumption $(H_0)$, the energy functional $J$ is well-defined on $H \times H$. Indeed, since the embedding ${H} \hookrightarrow L^p(\R^3)$ is continuous for $2 \leq p \leq 6$, and for $r_1, r_2>1, r_1+r_2 <6$ there is a $q>1$ satisfying $2 \leq r_1 q, r_2 q' \leq 6,$ where $ q':= \frac{q}{q-1}$, hence by using the H\"older inequality,
\begin{align} \label{holder}
\int_{\R^N} |u_1|^{r_1}|u_2|^{r_2} \, dx
   \leq \|u_1\|_{r_1q}^{r_1}\|u_2\|_{r_2q'}^{r_2}
   < \infty.
\end{align}
The well-known Gagliardo-Nirenberg inequality for $u \in H$ and $2 \le p \le 6$,
\begin{align} \label{GN}
\|u\|_p \le C(N,p)\|\nabla u\|_2^\alpha \|u\|_2^{1-\alpha},\quad
 \text{where}  \,\, \alpha:=\frac{N(p-2)}{2p},
\end{align}
which infers for $(u_1, u_2) \in S(a_1, a_2),$
\begin{align} \label{bd}
\int_{\R^N} |u_i|^{p_i} \, dx
     \le C(N,p_i,a_i) \|\nabla u_i\|_2^{\frac{N(p_i-2)}{2}} \ \ \text{for} \ \ i=1,2,
\end{align}
and
\begin{align} \label{mix}
 \int_{\R^N} |u_1|^{r_1}|u_2|^{r_2} \, dx
 \le C\|\nabla u_1\|_2^{\frac{N(r_1q-2)}{2q}}
  \|\nabla u_2\|_2^{\frac{N(r_2q'-2)}{2q'}},
\end{align}
where we used \eqref{holder} and $C=C(N,r_1,r_2,a_1,a_2,q)$.

In the following, for any $x \in \R^3$ we write $x:=(x', x_3)$ with $x':=(x_1, x_2) \in \R^2$ and $x_3 \in \R$. We now recall some results regarding rearrangement. We begin with introducing the Steiner rearrangement. Let $u: \R^3 \to \R$ be a Lebesgue measurable function and vanish at infinity, here $u$ is said to vanish at infinity if $\lim_{|x| \to \infty} u(x)=0$. For any $t >0$, $x'\in \R^2$, setting
$$
\{|u(x', y)| > t\}:=\{y \in \R: |u(x', y)| >t\},
$$
we then define $u^*$ the Steiner rearrangement of $u$ by
\begin{align} \label{rearrange}
u^*(x)=u^*(x', x_3):=\int_{0}^{\infty} \chi_{\{|u(x', y)| >t\}^*}(x_3)\, dt,
\end{align}
where $A^* \subset \R$ stands for the Steiner rearrangement of set $A \subset \R$ given by
$$
A^*:=(-\mathcal{L}^1(A)/2, \, \mathcal{L}^1(A)/2),
$$
here $\mathcal{L}^n(A)$ denotes the $n$-dimensional Lebesgue measure of set $A \subset \R^n$. In view of the definition \eqref{rearrange}, for any $x' \in \R^2$ we then see that the function $x_3 \mapsto u^*(x', x_3)$ is nonincreasing with respect to $|x_3|$, and $u^*(x', \cdot)$ is equimeasurable to $u(x', \cdot)$, namely for any $t >0$,
\begin{align} \label{measure1}
\mathcal{L}^1(\{u^*(x', y) >t\})=\mathcal{L}^1(\{|u(x', y)| >t\}).
\end{align}

We now collect some properties enjoyed by the Steiner rearrangement, whose proofs are similar to the ones of the classical Schwarz rearrangement, see \cite{Lieb}.

\begin{lem} \label{steiner}
Assume $u$ be a Lebesgue measurable function and vanish at infinity. Let $u^*$ be the Steiner rearrangement of $u$. Then
\begin{enumerate}
\item [(1)] $u^*$ and $|u|$ is equimeasurable in $\R^3$, i.e. for any $t >0$,
$$
\mathcal{L}^3(\{x \in \R^3: u^*(x)>t\})=\mathcal{L}^3(\{x \in \R^3: |u(x)|>t\});
$$
\item [(2)] $\int_{\R^3} |u^*|^p \, dx= \int_{\R^3} |u|^p \, dx$ \, for \, $1 \leq p < \infty$;
\item [(3)] If $u \in H^1(\R^3)$, then $u^* \in H^1(\R^3)$ and
$$
\int_{\R^3} |\partial_{x_i} u^*|^2 \, dx \leq \int_{\R^3}|\partial_{x_i} u|^2 \, dx \ \ \text{for} \ \ i=1,2,3;
$$
\item [(4)] Let $u, v$ be Lebesgue measurable functions and vanish at infinity, then
$$
\int_{\R^3}|u|^{r_1}|v|^{r_2} \, dx \leq \int_{\R^3}|u^*|^{r_1}|v^*|^{r_2} \, dx.
$$
\end{enumerate}
\end{lem}

We next introduce the coupled rearrangement due to Shibata \cite{Sh}, which can be viewed as an extension of the Steiner rearrangement. Let $u, v: \R^3 \to \R$ be Lebesgue measurable functions and vanish at infinity. For any $t >0, x' \in \R^2$, we define $u \ast v$ the coupled rearrangement of $u, v$ by
\begin{align} \label{defcoupled}
(u \ast v)(x)=(u \ast v)(x', x_3): =\int_{0}^{\infty} \chi_{\{|u(x', y)| >t\} \ast \{|v(x', y)| > t\}}(x_3) \, dt,
\end{align}
where $A \ast B \subset \R$ is the coupled rearrangement of sets $A, B \subset \R$ given by
$$
A \ast B:= (-(\mathcal{L}^1(A)+ \mathcal{L}^1(B))/2, \, (\mathcal{L}^1(A)+ \mathcal{L}^1(B))/2).
$$
Noting the definition \eqref{defcoupled}, for any $x' \in \R^2$ we immediately find that the function $x_3 \mapsto (u \ast v) (x', x_3)$ is nonincreasing with respect to $|x_3|$, and for any $t>0$,
\begin{align} \label{measure2}
\mathcal{L}^1(\{|(u \ast v)(x', y)| >t\})=\mathcal{L}^1(\{|u(x', y)| >t\})+ \mathcal{L}^1(\{|v(x', y)| >t\}).
\end{align}

We now summarize some facts concerning the coupled rearrangement.

\begin{lem} \label{coupled}
Assume $u, v$ are Lebesgue measurable functions and vanish at infinity. Let $u \ast v$ be the coupled rearrangement of $u , v$. Then
\begin{enumerate}
\item [(1)]
\begin{align*}
&\mathcal{L}^3(\{x \in \R^3: (u \ast v)(x) >t\})\\
&=\mathcal{L}^3(\{x \in \R^3: |u(x)| >t\}) + \mathcal{L}^3(\{x \in \R^3: |v(x)| >t\});
\end{align*}
\item [(2)] $\int_{\R^3} |u \ast v|^p \, dx = \int_{\R^3} |u| ^p + |v|^p \, dx$ \, for \, $1 \leq p < \infty$;
\item [(3)] If $u, v \in H^1(\R^3)$, then $u \ast v \in H^1(\R^3)$ and
$$
\int_{\R^3} |\partial_{x_i} (u \ast v)|^2 \, dx \leq \int_{\R^3}|\partial_{x_i} u|^2 + |\partial_{x_i} v|^2 \, dx
\ \ \text{for} \ \ i=1,2,3.
$$
In addition, if $u, v \in H^1(\R^3)\cap C^1(\R^3)$ are positive, and $x_3 \mapsto u(x', x_3), x_3 \mapsto v(x', x_3)$ are nonincreasing with respect to $|x_3|$, then
\begin{align}\label{strict}
\int_{\R^3}|\nabla (u \ast v)|^2 \, dx < \int_{\R^3} |\nabla u|^2 + |\nabla v|^2 \, dx;
\end{align}
\item [(4)] Let $u_1, u_2, v_1, v_2$ be Lebesgue measurable functions and vanish at infinity, then
$$
\int_{\R^3}|u_1|^{r_1}|u_2|^{r_2} + |v_1|^{r_1}|v_2|^{r_2} \, dx
\leq \int_{\R^3} \left(|u_1|^{r_1} \ast |v_1|^{r_2}\right)\left(|u_2|^{r_1} \ast |v_2|^{r_2}\right) \, dx.
$$
\end{enumerate}
\end{lem}
\begin{proof}
The assertions $(1)$-$(3)$ directly come from Lemma 2.2- Lemma 2.4 in \cite{Sh}. To establish assertion $(4)$, one can closely follow the approach as presented in the proof of \cite[Lemma A.2]{Sh} with some minor changes.
\end{proof}

An important nature of the coupled rearrangement is the strict inequality \eqref{strict}, which is somehow an extension of \cite[Lemma 2.10]{AB} to high dimension. Indeed, this property plays a crucially role to prevent dichotomy from occurring in the proof of Theorem \ref{existence}.

To end this section we show a Brezis-Lieb type result in $H \times H$, whose proof can be done by adapting the approach to prove \cite[Lemma 3.2]{GoJe}, see also \cite[Lemma 2.3]{ChZo}.

\begin{lem} \label{Brezis}
Assume $r_1, r_2 > 1, r_1 + r_2 < 6.$ If $(u_1^n, u_2^n)\rightharpoonup (u_1, u_2)$ in $H \times H$ as $n \to \infty$,
then
$$
\int_{\R^3}|u_1^n|^{r_1}|u_2^n|^{r_2} \, dx
= \int_{\R^3}|u_1|^{r_1}|u_2|^{r_2} + |u_1^n-u_1|^{r_1}|u_2^n-u_2|^{r_2}\,dx +o_n(1),
$$
where $o_n(1) \to 0$ as $n \to \infty$.
\end{lem}

\section{One constraint minimization problem} \label{scal}

In this section we deal with an one constraint minimization problem. Let $\mu >0$, $2 <p < \frac{10}{3}$, and define the energy functional $I_{\mu, p}: H \rightarrow \R$ by
$$
I_{\mu, p}(u):= \frac 12 \int_{\R^3}|\nabla u|^2  + (x_1^2 + x_2^2)|u|^2 \, dx -\frac{\mu}{p} \int_{\R^3}|u|^p \, dx.
$$
Clearly, by using \eqref{bd}, one can find that the energy functional $I_{\mu, p}$ restricted to $S(a)$ is bounded from below. Thus we introduce the following minimization problem
\begin{align} \label{smin}
m_{\mu, p}(a):= \inf_{u \in S(a)} I_{\mu, p}(u).
\end{align}
One can easily check that the function $a \mapsto m_{\mu, p}(a)$ is continuous for any $a \geq 0$. Our main result is the following.

\begin{thm} \label{scalar}
Assume $ \mu>0, 2 < p < \frac{10}{3}$. Then any minimizing sequence to \eqref{smin} is compact, up to translation, in $H$.
\end{thm}

In order to prove the compactness of any minimizing sequence to \eqref{smin}, up to translation, in $H$, we shall apply the Lions concentration compactness principle \cite{Li1, Li2}. Thus it is sufficient to exclude the possibilities of vanishing and dichotomy. Here vanishing is ruled out by taking advantage of \cite[Lemma 2.1]{Be}, see Lemma \ref{spectral}. While dichotomy is avoided by establishing the following strict subadditivity inequality
$$
m_{\mu, p}(a)< m_{\mu, p}(a-a_0) + m_{\mu, p}(a_0) \ \ \text{for} \ \ 0<a_0<a.
$$

\begin{lem} \cite[Lemma 2.1]{Be} \label{spectral}
Define
\begin{align} \label{La}
\Lambda_0 := \inf_{\int_{\R^3}|u|^2 \, dx =1} \int_{\R^3}|\nabla u|^2 + (x_1^2 + x_2^2) |u|^2 \, dx,
\end{align}
and
\begin{align} \label{la}
\lambda_0:= \inf_{\int_{\R^2}|v|^2 \, dx' =1} \int_{\R^2}| \partial_{x_1} v|^2 +|\partial_{x_2} v|^2 + (x_1^2 + x_2^2) |v|^2 \, dx'.
\end{align}
Then $\Lambda_0 =\lambda_0$.
\end{lem}

With the help of the previous lemma, we are able to prove that any minimizing sequence to \eqref{smin} does not vanish. More precisely, we have the following result.

\begin{lem} \label{nonvanish}
Let $ \mu>0, 2 < p < \frac{10}{3}$. Assume $\{u_n\} \subset S(a)$ is a minimizing sequence to \eqref{smin}. Then there exists $\delta>0$ such that
$$
 \lim_{n \to \infty} \int_{\R^3}|u_n|^p \, dx \geq  \delta.
$$
\end{lem}
\begin{proof}
Suppose by contradiction that $\int_{\R^3}|u_n|^p \, dx=o_n(1)$, thus we have that
\begin{align} \label{minimizing}
\begin{split}
m_{\mu, p}(a)= I_{\mu, p}(u_n) + o_n(1)&=\frac 12 \int_{\R^3}|\nabla u_n|^2  + (x_1^2 + x_2^2)|u_n|^2 \, dx + o_n(1)\\
&\geq \frac {a \Lambda_0}{2} + o_n(1),
\end{split}
\end{align}
where $\Lambda_0$ is the constant defined by \eqref{La}. Noting the definition \eqref{la}, since the Sobolev space $\mathcal{H}:=\{w \in H^1(\R^2): \int_{\R^2}\left(x_1^2 + x_2^2 \right)|w|^2 \, dx < \infty \}$ is compactly embedded in $L^2(\R^2)$, thus it is standard to show that $\lambda_0$ is achieved by some $w \in H^1(\R^2)$ with $ \int_{\R^2} |w|^2 \, dx'=1$. Let $\varphi \in H^1(\R)$ satisfy $\int_{\R} |\varphi|^2 \, dx_3 =a$, and any $\lambda >0$ set
$$
u_{\lambda}(x):=w(x')\varphi_{\lambda}(x_3) \in S(a), \ \text{where} \ \varphi_{\lambda}(x_3):= \lambda^{\frac 12} \varphi(\lambda x_3).
$$
Hence by Lemma \ref{spectral},
\begin{align} \label{add}
\begin{split}
I_{\mu, p}(u_{\lambda}) &= \frac{a \lambda_0}{2} + \frac 1 2 \int_{\R}|\partial_{x_3} \varphi_{\lambda}|^2 \, dx - \frac{\mu}{p} \int_{\R^2}|w|^p \, dx' \int_{\R} |\varphi_{\lambda}|^p \, dx_3 \\
&=\frac{a \Lambda_0}{2} + \frac {\lambda^2} {2} \int_{\R}|\partial_{x_3} \varphi|^2 \, dx_3 -  \lambda^{\frac p2-1}\frac{\mu}{p} \int_{\R^2}|w|^p \, dx' \int_{\R} |\varphi|^p \, dx_3.\\
\end{split}
\end{align}
Since $2 <p< \frac{10}{3}$, it follows from \eqref{add} that $m_{\mu, p}(a) \leq I_{\mu, p}(u) < \frac{a \Lambda_0}{2}$ for $\lambda>0$ small enough. This obviously contradicts \eqref{minimizing}, and ends the proof.
\end{proof}

\begin{lem} \label{nonzero}
Let $ \mu>0, 2 < p < \frac{10}{3}$. Assume $\{u_n\} \subset S(a)$ is a minimizing sequence to \eqref{smin}. Then there exist a sequence $\{k_n\} \subset \R$ and $u \in H \backslash\{0\}$ such that $u_n(x', x_3-k_n) \rightharpoonup u$ in $H$ as $n \to \infty$.
\end{lem}
\begin{proof}
From \eqref{bd} we know that $\{u_n\}$ is bounded in $H$. Then by Lemma \ref{nonvanish} and the interpolation inequality in Lebesgue space, we find that there is a constant $\delta_1 >0$ such that
\begin{align} \label{low}
\lim_{n \to \infty}\int_{\R^3}|u_n|^{\frac {10}{3}} \, dx \geq \delta_1.
\end{align}
We now apply the Gagliardo-Nirenberg inequality on $T_k$, here $T_k:=\R^2 \times [k, k+1)$ for any $k \in \N$, then
\begin{align*}
\int_{T_k}|u_n|^{\frac{10}{3}} \, dx \leq C \left( \int_{T_k}|u_n|^2 \, dx \right)^{\frac 2 3} \|u_n\|_{\dot{H}(T_k)}^2.
\end{align*}
Notice that $\R^3=\cup_{-\infty}^{\infty}\R^2 \times [k, k+1)$, then by summing above inequality with respect to $k \in \N$ there holds
\begin{align}\label{sumk}
\delta_1 \leq \int_{\R^3}|u_n|^{\frac{10}{3}} \, dx \leq C \left( \sup_{k \in \Z} \int_{T_k}|u_n|^2 \, dx \right)^{\frac 2 3} \|u_n\|_{\dot{H}}^2.
\end{align}
Recall that $\{u_n\}$ is bounded in $H$, then \eqref{sumk} yields that for any $n \in \N$ there admits $k_n \in \N$ such that
$$
\int_{T_{k_n}}|u_n|^2\, dx \geq \delta_2
$$
holds for some $\delta_2 >0$. We now set $w_n(x):=u_n(x', x_3-k_n)$, hence $\int_{T_0} |w_n|^2 \, dx \geq \delta_2.$ By using the fact that the embedding $H^1(T_0)\hookrightarrow L^2(T_0)$ is compact, there exists $u \in H \backslash\{0\}$ such that $w_n \rightharpoonup u$ in $H$ as $n \to \infty$. Thus the proof is completed.
\end{proof}

\begin{proof} [Proof of Theorem \ref{scalar}]
Let $\{u_n\} \subset S(a)$ be an arbitrary minimizing sequence to \eqref{smin}. In light of Lemma \ref{nonzero}, we know that there is a sequence $\{k_n\} \subset \R$ so that $w_n(x):=u_n(x', x_3-k_n)$ admits a nontrivial weak limit $u$ in $H$. Observe that $\{w_n\} \subset S(a)$ is also a minimizing sequence to \eqref{smin}. To see the compactness of $\{w_n\}$, let us first deduce that $u \in S(a)$. To this end, we argue by contradiction that $0 < a_0:=\|u\|_{2}^2 < a$. By the Brezis-Lieb Lemma \cite{BrLi},
\begin{align*}
\begin{split}
\|w_n\|_{\dot H}^2 &=\|w_n-u\|_{\dot H}^2+ \|u\|_{\dot H}^2+o_n(1),\\
\|w_n\|_{q}^q & = \|w_n- u\|_{q}^q + \|u\|_{q}^q + o_n(1) \ \ \text{for} \ \ 1  \leq  q < \infty.
\end{split}
\end{align*}
Hence
\begin{align*}
I_{\mu, p}(w_n)= I_{\mu, p}(w_n-u) + I_{\mu, p}(u) + o_n(1),
\end{align*}
which gives that
\begin{align} \label{subin}
m_{\mu, p}(a) \geq m_{\mu, p}(a-a_0) + m_{\mu, p}(a_0),
\end{align}
where we used the continuity of the function $a \mapsto m_{\mu, p}(a)$ with respect to any $a \geq 0$.

Next we claim that
\begin{align} \label{subadd}
m_{\mu, p}(\theta a) < \theta m_{\mu, p}(a) \ \ \mbox{for any } \ \theta >1.
\end{align}
Indeed, we define $v_n:= \theta^{\frac 12} w_n \in S(\theta a)$, then
\begin{align*}
m_{\mu, p}(\theta a) \leq I_{\mu, p}(v_n) & = \frac {\theta}{2} \int_{\R^2}|\nabla w_n|^2  + (x_1^2 + x_2^2)|w_n|^2 \, dx - \theta^{\frac p2} \frac{\mu}{p} \int_{\R^3} |w_n|^p \, dx \\
& = \theta I_{\mu, p}(u_n) + \frac{\mu}{p}\left(\theta- \theta^{\frac p 2}\right) \int_{\R^3} |u_n|^p \, dx.
\end{align*}
Due to $\theta >1, p>2$, it then follows from Lemma \ref{nonvanish} that \eqref{subadd} necessarily holds.
Now by using \eqref{subadd}, we get that
$$
m_{\mu, p}(a)= \frac{a-a_0}{a}m_{\mu, p}(a) + \frac {a_0}{a}m_{\mu, p}(a) < m_{\mu, p}(a-a_0) + m_{\mu, p}(a_0),
$$
which contradicts \eqref{subin}. This in turn gives that $u \in S(a)$. Thus from the interpolation inequality in Lebesgue space we know that $w_n \to u$ in $L^q(\R^N)$ for $2 \leq q <6$ as $n \to \infty$. Applying the weak lower semicontinuity of the norm, we then find that $I_{\mu, p}(u) \leq m_{\mu, p}(a)$. Since $m_{\mu, p}(a)\leq I_{\mu, p}(u)$, it follows that $ I_{\mu, p}(u)=m_{\mu, p}(a)=I_{\mu, p}(w_n)+o_n(1),$ therefore we have that $w_n \to u $ in $H$ as $n \to \infty$. Thus we finish the proof.
\end{proof}

\section{Proof of Theorem \ref{existence}} \label{proof}

In this section we prove Theorem \ref{existence}. To do this, we first collect some basic properties with respect to the function $(a_1, a_2) \mapsto m(a_1, a_2)$.

\begin{lem} \label{prop}
Assume $a_1, a_2 \geq 0$. Then
\begin{enumerate}
\item [(1)] $(a_1, a_2) \mapsto m(a_1, a_2)$ is continuous;
\item [(2)] $ m(a_1, a_2) \leq m(b_1, b_2) + m(c_1, c_2)$, where $a_i=b_i + c_i$ with $b_i, c_i \geq 0$ for $i=1, 2$;
\item [(3)]$m(a_1, a_2) < \frac{\Lambda_0 a_1}{2} + m_{\mu_2, p_2}(a_2)$ and $m(a_1, a_2) < m_{\mu_1, p_1}(a_1) + \frac{\Lambda_0 a_2}{2}$,
where $m_{\mu, p}(a)$ and $\Lambda_0$ are defined by \eqref{smin} and \eqref{La}, respectively.
\end{enumerate}
\end{lem}
\begin{proof}
$(1)$ For any $(a_1, a_2) \in {\R}^+ \times  {\R}^+$, suppose $\{(a_1^n, a_2^n)\} \subset {\R}^+ \times {\R}^+$ is a sequence satisfying $(a_1^n, a_2^n) \to (a_1, a_2)$ as $n \to \infty$, where ${\R}^+:=(0, \infty)$. By the definition \eqref{min}, for any $\eps >0$ there exists a sequence $\{(u_1^n, u_2^n)\} \subset S(a_1^n) \times S(a_2^n)$ such that
\begin{align} \label{conti}
J(u_1^n, u_2^n) \leq m(a_1^n, a_2^n) + \frac{\eps}{2}.
\end{align}
We now set
$$
v_i^n:=\left(\frac{a_i}{a_i^n}\right)^{\frac 12}u_i^n \in S(a_i) \ \ \text{for}  \ i=1, 2.
$$
Since $(a_1^n, a_2^n) \to (a_1, a_2)$ as $n \to \infty,$ it then follows from \eqref{conti} that for $n \in \N $ sufficiently large
\begin{align*}
m(a_1, a_2) \leq J(v_1^n, v_2^n) \leq m(a_1^n, a_2^n) + \eps.
\end{align*}
On the other hand, we similarly get that for $n \in \N $ sufficiently large
\begin{align*}
m(a_1^n, a_2^n)\leq m(a_1, a_2) + \eps.
\end{align*}
Hence $m(a_1^n, a_2^n) \to m(a_1, a_2)$ as $n \to \infty$. This then shows that the function $(a_1, a_2) \mapsto m(a_1, a_2)$ is continuous.

$(2)$ Recalling the definition \eqref{min}, for any $\eps >0$ there are $(u_1, u_2) \in S(b_1) \times S(b_2)$ and $(v_1, v_2) \in S(c_1) \times S(c_2)$ such that
$$
J(u_1, u_2) \leq m(b_1, b_2) + \frac \eps 2 \ \ \text{and} \ \ J(v_1, v_2) \leq m(c_1, c_2) + \frac \eps 2.
$$
Without loss of generality, we assume that $supp \, u_i \cap supp \, v_i = \emptyset$ for $i=1, 2$. Thus defining $ w_i:=u_i + v_i \in S(a_i)$, we have that
$$
m(a_1, a_2) \leq J(w_1, w_2) \leq J(u_1, u_2)+ J(v_1, v_2) \leq m(b_1, b_2) + m(c_1, c_2) + \eps,
$$
from which we deduces that $m(a_1, a_2) \leq m(b_1, b_2) + m(c_1, c_2).$

$(3)$ Observing the definition \eqref{la}, we know that $\lambda_0$ is achieved by some $w \in H^1(\R^2)$ satisfying $\int_{\R^2}|w|^2 \, dx' =1$. Let $\varphi \in H^1(\R)$ be such that $\int_{\R}|\varphi|^2 \, dx_3=a_1$. We then write $u_1(x):=w(x')\varphi_{\lambda}(x_3) \in S(a_1)$, where $\varphi_{\lambda}(x_3):= \lambda^{\frac 12} \varphi(\lambda x_3)$ for $\lambda >0$. On the other hand, from Theorem \ref{scalar} we know that there exists a $u_2 \in S(a_2)$ such that $I_{\mu_2, p_2} (u_2)=m_{\mu_2, p_2}(a_2)$. Notice that
\begin{align*}
J(u_1, u_2) \leq \frac{\Lambda_0 a_1} {2} + m_{\mu_2, p_2}(a_2) + \frac {\lambda^2} 2 \int_{\R}|\partial_{x_3} \varphi|^2 \, dx
- \lambda^{\frac {p_1} {2}-1}\frac{\mu_1}{p_1} \int_{\R^2}|w|^{p_1} \, dx' \int_{\R} |\varphi|^{p_1} \, dx_3.
\end{align*}
Because $2 < p_1 < \frac{10}{3}$, then $m(a_1, a_2) \leq J(u_1, u_2) < \frac{\Lambda_0 a_1} {2} + m_{\mu_2, p_2}(a_2)$ for $\lambda>0$ sufficiently small. Furthermore, by the same arguments we can obtain that $m(a_1, a_2) < m_{\mu_1, p_1}(a_1) + \frac{\Lambda_0 a_2}{2}$. Thus we finish the proof.
\end{proof}

Now we are in a position to prove Theorem \ref{existence}.

\begin{proof}[Proof of Theorem \ref{existence}] Let $\{(u_1^n, u_2^n)\} \subset S(a_1, a_2)$ be an arbitrary minimizing sequence to \eqref{min}. Firstly, in view of the assumption $(H_0)$, from \eqref{bd}-\eqref{mix} we then know that $\{(u_1^n, u_2^n)\}$ is bounded in $H \times H$. We shall proceed in three steps to prove the compactness of $\{(u_1^n, u_2^n)\}$, up to translation, in $H \times H$. \medskip \\
{\bf Step 1:} We claim that there is a $\delta >0$ such that
\begin{align} \label{lp}
\lim_{n \to \infty}\int_{\R^3} |u_1^n|^{p_1} \, dx \geq \delta \ \ \text{and} \ \ \lim_{n \to \infty}\int_{\R^3} |u_2^n|^{p_2} \, dx  \geq \delta.
\end{align}

Arguing by contradiction, without restriction we assume that $\int_{\R^3} |u_1^n|^{p_1} \, dx =o_n(1)$. Since $\{(u_1^n, u_2^n)\}$ is bounded in $H \times H$, then by using the interpolation inequality in Lebesgue space, it comes from \eqref{holder} that $\int_{\R^3}|u_1^n|^{r_1}|u_2^n|^{r_2}\, dx =o_n(1)$. Thus
\begin{align*}
 J(u_1^n, u_2^n) \geq \frac{\Lambda_0 a_1}{2} + m_{\mu_2, p_2}(a_2)+ o_n(1),
\end{align*}
which means that
\begin{align*}
m(a_1, a_2) \geq \frac{\Lambda_0 a_1}{2} + m_{\mu_2, p_2}(a_2).
\end{align*}
This is a contradiction to the assertion $(3)$ in Lemma \ref{prop}. Therefore \eqref{lp} holds.

Next by applying the same arguments as presented in the proof Lemma \ref{nonzero}, it follows from \eqref{lp} that there are a sequence $\{y_1^n\} \subset \R$ and $u_1 \in H \backslash \{0\}$ such that $u_1^n(x', x_3-y_1^n) \rightharpoonup u_1$ in $H$ as $n \to \infty$. Note that $\{u_2^n\}$ is bounded in $H$, then there is $u_2 \in H$ such that $u_2^n(x', x_3-y_1^n) \rightharpoonup u_2$ in $H$ as $n \to \infty$. We now define $ v_1^n(x):=u_1^n(x', x_3-y_1^n)-u_1(x), v_2^n(x):=u_2^n(x', x_3-y_1^n)-u_2(x),$ $0<b_1:=\|u_1\|_{2}^2 \leq a_1$ and $0 \leq b_2:=\|u_2\|_{2}^2 \leq a_2$. \medskip \\
{\bf Step 2:} We claim that $b_1=a_1$, namely $u_1 \in S(a_1)$.

We suppose by contradiction that $b_1 < a_1$. From the Brezis-Lieb Lemma, then
\begin{align*}
\begin{split}
\|(u_1^n, u_2^n)\|_{\dot H \times \dot H}^2 &=\|(v_1^n, v_2^n)\|_{\dot H \times \dot H}^2+ \|(u_1, u_2)\|_{\dot H \times \dot H}^2+o_n(1),\\
\|(u_1^n, u_2^n)\|_{q}^q & = \|(v_1^n, v_2^n)\|_{q}^q + \|(u_1, u_2)\|_{q}^q + o_n(1) \ \ \text{for} \ \ 1  \leq  q < \infty.
\end{split}
\end{align*}
This joints with Lemma \ref{Brezis} and the assertion $(1)$ in Lemma \ref{prop} assert that
\begin{align} \label{global}
\begin{split}
J(u_1^n, u_2^n)&=J(u_1, u_2) + J(v_1^n, v_2^n) + o_n(1) \\
&  \geq J(u_1, u_2) + m(a_1-b_1, a_2-b_2) + o_n(1).
\end{split}
\end{align}
Thus
$$
m(a_1, a_2) \geq m(b_1, b_2) + m(a_1-b_1, a_2-b_2).
$$
It then results from the assertion $(2)$ in Lemma \ref{prop} that
\begin{align} \label{subadd1}
m(a_1, a_2)=m(b_1, b_2) + m(a_1-b_1, a_2-b_2).
\end{align}
Since $J(u_1^n, u_2^n)=m(a_1, a_2)+o_n(1)$, by using \eqref{global}-\eqref{subadd1}, there holds
\begin{align} \label{minineq1}
J(u_1, u_2)= m(b_1, b_2).
\end{align}

In order to obtain a contradiction we now consider two cases: $b_2 < a_2$ and $b_2 = a_2$. We first deal with the case $b_2 < a_2$. In this case we assume $\{(w_1^n, w_2^n)\} \subset S(a_1-b_1, a_2-b_2)$ is a minimizing sequence to $m(a_1-b_1, a_2-b_2)$. Reasoning as before, we conclude that there exist a sequence $\{z_1^n\} \subset \R$, $w_1 \in H \backslash \{0\}$ and $ w_2 \in H$ such that $ w_1^n(x', x_3-z_n^1) \rightharpoonup w_1$ and $ w_2^n(x', x_3-z_n^1) \rightharpoonup w_2 \ \text{in} \ H \ \text{as}\ n \to \infty.$ Furthermore, defining $c_1 :=\|w_1\|_{2}^2 >0$ and $c_2:=\|w_2\|_{2}^2 \geq 0$, we have that
\begin{align} \label{subadd2}
m(a_1-b_1, a_2-b_2)=m(c_1, c_2) + m(a_1-b_1-c_1, a_2-b_2-c_2),
\end{align}
as well as
\begin{align} \label{minineq2}
J(w_1, w_2)=m(c_1, c_2).
\end{align}
Now combining \eqref{subadd1} with \eqref{subadd2}, we see that
\begin{align} \label{aim}
m(a_1, a_2)=m(b_1, b_2) + m(c_1, c_2) + m(a_1-b_1-c_1, a_2-b_2-c_2).
\end{align}
Noticing \eqref{measure1}, we get that
\begin{align}\label{norm}
\int_{\R^3}(x_1^2 + x_2^2)|u^*|^2 \, dx = \int_{\R^2} (x_1^2 + x_2^2) \int_{\R} |u^*|^2 \, dx_3 dx'= \int_{\R^3}(x_1^2 + x_2^2)|u|^2 \, dx,
\end{align}
then as a result of the assertions $(2)$-$(4)$ in Lemma \ref{steiner}, we deduce from \eqref{minineq1} and \eqref{minineq2} that
\begin{align} \label{equal}
J(u_1^*, u_2^*)=m(b_1, b_2) \ \ \text{and} \ \  J(w_1^*, w_2^*)=m(c_1, c_2).
\end{align}
This implies that $(u_1^*, u_2^*)$ and $(w_1^*, w_2^*)$ are solutions to \eqref{system} with some $(\lambda_1, \lambda_2) \in \R^2$. Hence by the elliptic regularity theory we have that $u_i^*, w_i^* \in C^2(\R^3)$. Moreover, if $b_2, c_2 >0$, the maximum principle reveals that $u_i^*, w^*_i >0$ for $i=1 ,2$. We now apply \eqref{measure2}, thus
\begin{align*}
\int_{\R^3} (x_1^2 + x_2^2)|u_i^* \ast w_i^*|^2 \, dx &= \int_{\R^2} (x_1^2 + x_2^2) \int_{\R} |u_i^* \ast w_i^*|^2 \, dx_3 dx'\\
&=\int_{\R^3} (x_1^2 + x_2^2)(|u_i^*|^2 + |w_i^*|^2) \, dx
\end{align*}
for $i=1, 2$. By using the assertions $(2)$-$(4)$ in Lemma \ref{coupled} and \eqref{equal}, we then find that
\begin{align*}
m(b_1+ c_1, b_2 + c_2) &\leq J(u_1^* \ast w_1^*, u_2^* \ast w_2^*) < J(u_1^*, u_2^*) + J(w_1^*, w_2^*) \\
&= m(b_1, b_2) + m(c_1, c_2).
\end{align*}
If $b_2=0$ or $c_2 =0$, since we have that $u_1^*, w_1^* >0$, thus the above strict inequality remain valid. Hence by the assertion $(2)$ in Lemma \ref{prop}, we reach a contradiction from \eqref{aim}. Next we consider the case that $b_2=a_2$. In this case we can also obtain a contradiction through the same arguments as handling the case $b_2< a_2$. As a consequence, we have proved that $b_1=a_1$. Therefore
\begin{align}\label{l21}
u_1^n(x', x_3-y_n^1) \to {u}_1(x) \ \text{in} \ L^2(\R^3) \ \text{as}\ n \to \infty.
\end{align}
We now reverse the roles of $\{u_1^n\}$ and $\{u_2^n\}$ in Step 1- Step 2, then there exist a sequence $\{y_2^n\} \subset \R$, $\tilde{u}_1 \in H$ and $ \tilde{u}_2 \in S(a_2)$ such that
$u_1^n(x', x_3-y_n^2) \rightharpoonup \tilde{u}_1(x)$ and $ u_2^n(x', x_3-y_n^2) \rightharpoonup \tilde{u}_2(x) \ \text{in} \ H \ \text{as}\ n \to \infty.$
Moreover,
\begin{align}\label{l22}
u_2^n(x', x_3-y_n^2) \to \tilde{u}_2(x) \ \text{in} \ L^2(\R^3) \ \text{as}\ n \to \infty.
\end{align}
{\bf Step 3:} We claim that
\begin{align} \label{lim}
\lim_{n \to \infty}|y_1^n-y_2^n| < \infty.
\end{align}

If this were false, we would get that $\tilde{u}_1=u_2=0$ due to $(u_1, \tilde{u}_2) \in S(a_1, a_2)$. Then from the Brezis-Lieb Lemma and Lemma \ref{Brezis},
\begin{align*}
&J(u_1^n ,u_2^n)=J(u_1^n(x', x_3-y_1^n), u_2^n(x', x_3-y_1^n)) \\
&=J(u_1, 0)+ J(u_1^n(x', x_3-y_1^n)-u_1, u_2^n(x', x_3-y_1^n))+ o_n(1) \\
&=J(u_1, 0)+ J(u_1^n(x', x_3-y_2^n)-u_1(x', x_3+y_1^n-y_2^n), u_2^n(x', x_3-y_2^n))+ o_n(1) \\
&=J(u_1, 0)+ J(0, \tilde{u}_2) + J(u_1^n(x', x_3-y_2^n)-u_1(x', x_3+y_1^n-y_2^n), u_2^n(x', x_3-y_2^n)-\tilde{u}_2) + o_n(1).
\end{align*}
Hence it infers from \eqref{l21}-\eqref{l22} that
\begin{align*}
m(a_1, a_2)  \geq J(u_1, 0) + J(0, \tilde{u}_2).
\end{align*}
According to \eqref{norm} and the assertions $(2)$-$(3)$ in Lemma \ref{steiner}, we then obtain that
$$
m(a_1, a_2) \geq  J(u_1^*, 0) + J(0, \tilde{u}_2^*) \geq m_{\mu_1, p_1}(a_1) + m_{\mu_1, p_1}(a_1).
$$
By using the assertion $(2)$ in Lemma \ref{prop}, we get that
\begin{align} \label{subadd3}
m(a_1, a_2) =  J(u_1^*, 0) + J(0, \tilde{u}_2^*).
\end{align}
Record that $(u_1^*, \tilde{u}_2^*) \in S(a_1, a_2)$, then
$$
m(a_1, a_2) \leq J(u_1^*,  \tilde{u}_2^*) < J(u_1^*, 0) + J(0, \tilde{u}_2^*),
$$
which contradicts \eqref{subadd3}. Therefore \eqref{lim} holds, which gives that there exists $y \in \R$ such that $y_1^n=y_2^n + y +o_n(1),$ up to translation.
We now define
$$
\varphi_1^n(x):=u_1^n(x', x_3-y_1^n), \quad \varphi_2^n(x):=u_2^n(x', x_3-y_1^n).
$$
Thus there is $(\varphi_1(x), \varphi_2(x)):=(u_1(x), \tilde{u}_2(x', x_3-y)) \in S(a_1, a_2)$ such that $(\varphi_1^n, \varphi_2^n) \rightharpoonup (\varphi_1, \varphi_2)$ in $H \times H$ as $n \to \infty.$
Furthermore, we have that $(\varphi_1^n, \varphi_2^n) \to (\varphi_1, \varphi_2)$ in $L^{p}(\R^3) \times L^{p}(\R^3)$ as $n \to \infty$ for $2 \leq p <6$. By consequence,
\begin{align*}
J(u_1^n, u_2^n)= J(\varphi_1^n, \varphi_2^n) &=J(\varphi_1, \varphi_2) + J(\varphi_1^n - \varphi_1, \varphi_2^n-\varphi_2) + o_n(1) \\
&\geq m(a_1, a_2)
 + \frac12 \|(\varphi_1^n-\varphi_1, \varphi_2^n-\varphi_2)\|^2_{\dot H \times \dot H}+o_n(1).
\end{align*}
Since $J(u_1^n, u_2^n)=m(a_1, a_2)+o_n(1)$, we then obtain that $(\varphi_1^n, \varphi_2^n) \to (\varphi_1, \varphi_2)$ in $H \times H$ as $n \to \infty$. Thus we end the proof.
\end{proof}

{\sc Address of the author:}\\[1em]
\begin{tabular}{ll}
Tianxiang Gou \\
Laboratoire de Math\'ematiques (UMR 6623), \\
Universit\'e de Bourgogne Franche-Comt\'e, \\
16, Route de Gray 25030 Besan\c{c}on Cedex, France.\\
School of Mathematics and Statistics, \\
Lanzhou University, Lanzhou, Gansu 730000, People's Republic of China.\\
goutx14@lzu.edu.cn
\end{tabular}

\end{document}